\documentclass{article}
\usepackage{amsmath,amssymb,amsthm}
\makeatletter
\let\@fnsymbol\@arabic
\makeatother

\usepackage{authblk}
\title{Another article on the number of homomorphisms}
\author[ ]{Alexander V. Khudyakov}
\affil[ ]{\textit{Faculty of Mechanics and Mathematics of Moscow State University,
Moscow 119991, Leninskie gory, MSU.}}
\affil[ ]{\textit{Moscow Center for Fundamental and Applied Mathematics}}
\affil[ ]{\textit{forsovenok121212@yandex.ru}}

\newcommand\blfootnote[1]{
    \begingroup
    \renewcommand\thefootnote{}\footnote{#1}
    \addtocounter{footnote}{-1}
    \endgroup
}

\let\ge\geqslant
\let\epi\twoheadrightarrow

\DeclareMathOperator{\Hom}{Hom}
\DeclareMathOperator{\lcm}{lcm}

\newcommand{\Z}{\mathbb{Z}}

\theoremstyle{plain}
\newtheorem{theorem}{Theorem}[section]
\newtheorem{corollary}[theorem]{Corollary}
\newtheorem{lemma}[theorem]{Lemma}
\newtheorem{observation}[theorem]{Observation}

\newtheorem*{lemma*}{Lemma}
\newtheorem*{claim*}{Claim}
\newtheorem*{conjecture*}{Conjecture}

\newtheorem*{asai_yoshida_conjecture}{Asai-Yoshida conjecture \cite{Asai_Yoshida_1993}}
\newtheorem*{bkv_theorem}{BKV theorem \cite{Brusyanskaya_2019}}
\newtheorem*{brauer_lemma}{Brauer lemma \cite{Brauer_1969}}
\newtheorem*{frobenius_theorem}{Frobenius theorem \cite{Frobenius_1895}}
\newtheorem*{restricted_asai_yoshida_conjecture}{Restricted Asai-Yoshida conjecture}
\newtheorem*{solomon_theorem}{Solomon theorem \cite{Solomon_1969}}
\newtheorem*{yoshida_theorem}{Yoshida theorem \cite{Yoshida_1993}}

\theoremstyle{definition}
\newtheorem{definition}[theorem]{Definition}
\newtheorem{example}[theorem]{Example}
\newtheorem{remark}[theorem]{Remark}

\begin{document}

\maketitle
\blfootnote{This work was supported by the Russian Science Foundation, project no. 22-11-00075.}

\begin{abstract}
    We extend the class of abelian groups for which a conjecture of Asai and Yoshida on the number of crossed homomorphisms holds. We also prove a general result which connects certain problems concerning divisibility in groups to the Asai-Yoshida conjecture. One of the consequences is that for finite groups \(F\) and \(G\) the number \(|\Hom (F, G)|\) is divisible by \(\gcd(|G|, |F : F'|)\) if \(F/F'\) is a product of a cyclic group and a group with cube-free exponent.
\end{abstract}

\section{Introduction}

The following two theorems, due to Frobenius and Solomon, may serve as a good illustration of the scope of problems we study.

\begin{frobenius_theorem}
    The number of solutions to the equation \(x^n = 1\) in a finite group \(G\) is divisible by \(\gcd(|G|, n)\).
\end{frobenius_theorem}

\begin{solomon_theorem}
    In any group, the number of solutions to a system of coefficient-free equations is divisible by the order of the group provided the number of equations is less than the number of unknowns.
\end{solomon_theorem}

There is a generalization of the Frobenius theorem due to Yoshida.

\begin{yoshida_theorem}
    The number of homomorphisms from a finite abelian group \(M\) to a finite group \(G\) is divisible by \(\gcd(|G|, |M|)\).
\end{yoshida_theorem}

Asai and Yoshida conjectured that a more general statement hold.

\begin{conjecture*}[\cite{Asai_Yoshida_1993}]
    The number of homomorphisms from a (not necessarily abelian) finite group \(F\) to a finite group \(G\) is divisible by \(\gcd(|G|, |F : F'|)\).
\end{conjecture*}

This conjecture is known to hold (see \cite{ASAI_TAKEGAHARA_1999}) when each \(p\)-torsion of a group \(F/F'\) is of the form \(\Z/p^n\Z \times (\Z/p\Z)^m\). Our Theorem~\ref{thm:annotated} (which can be concidered as the main result of the present paper) states that it also holds when those \(p\)-torsions are of the form \(\Z/p^n\Z \times (\Z/p\Z)^m \times (\Z/p^2\Z)^k\).

Asai and Yoshida used crossed homomorphisms to study that conjecture. They showed that it is true if, for each \(p\)-torsion \(M\) of \(F/F'\) and each \(p\)-subgroup \(H\) of \(G\), the following conjecture on the number of crossed homomorphisms holds (we recall the definition of a crossed homomorphism in the section ``Crossed homomorphisms").

\begin{asai_yoshida_conjecture}
    The number of crossed homomorphisms from a finite abelian group \(M\) to a finite group \(H\) on which \(M\) acts is divisible by \(\gcd(|H|, |M|)\).
\end{asai_yoshida_conjecture}

This conjecture has been verified in various cases (\cite{ASAI_TAKEGAHARA_1999} for abelian \(H\), \cite{Asai_Niwasaki_Takegahara_2003} when \(H\) is a semidihedral, generalized quaternion or dihedral \(2\)-group and \(M\) is of rank \(2\), \cite{Asai_2023} for \(M = \Z/p^n\Z \times (\Z/p\Z)^m \times  \Z/p^2\Z\); the first result of this kind was obtained by P.~Hall in \cite{Hall_1936} for cyclic \(M\)). We prove the conjecture for \(M = \Z/p^n\Z \times (\Z/p\Z)^m \times  (\Z/p^2\Z)^k\), and also when \(H\) is a \(p\)-group all of whose non-abelian subgroups have elementary abelian center. Such groups include, for instance, semidihedral, generalized quaternion and dihedral \(2\)-groups.

On the other side, a general result which implies both the Frobenius and Solomon theorems (but not the Yoshida theorem) was proven in \cite{Klyachko_2017} (for \(n = 0\)) and \cite{Brusyanskaya_2019} (for other values of \(n\)). A group \(F\) equipped with an epimorphism \(F \to \Z/n\Z\) (where \(n \in \Z\)) is called an \textit{\(n\)-indexed} group. This epimorphism \(F \to \Z/n\Z\) is called \textit{degree} and denoted \(\deg\).

\begin{bkv_theorem}

Suppose that an integer \(n\) is a multiple of the order of a subgroup \(H\) of a group \(G\) and a set \(\Phi\) of homomorphisms from an \(n\)-indexed group \(F\) to \(G\) satisfies the following conditions.
\begin{enumerate}

    \item \(\Phi\) is invariant with respect to conjugation by elements of \(H\).

    \item For any \(\phi \in \Phi\), each homomorphism \(\psi\) such that
    \[
    \left\{\begin{aligned}
        &\psi(f) = \phi(f) &\text{ for each element \(f \in F\) of degree zero,} \\
        &\psi(f) \in \phi(f)H &\text{ for each element \(f \in F\).}
    \end{aligned}\right.
    \]
    belongs to \(\Phi\) too.
    
\end{enumerate}
Then \(|\Phi|\) is divisible by \(|H|\).

\end{bkv_theorem}

We refer to \cite{Klyachko_2017}, \cite{Brusyanskaya_2019}, \cite{Klyachko_2020}, \cite{Brusyanskaya_2022}, \cite{Brusyanskaya_2024} for the various consequences of this theorem. It is natural to ask whether it is possible to replace a group indexed by \(\Z/n\Z\) in this theorem with a group indexed by an arbitrary finitely generated abelian group \(M\) of order \(n\) (in particular, a positive answer would imply that \(|\Hom(F, G)|\) is divisible by \(\gcd(|F : F'|, |G|)\) for finite groups \(F\) and \(G\)). Our Theorem~\ref{thm:hom_family} shows that such a generalization is possible precisely when the Asai-Yoshida conjecture holds for the group \(M\).

We recall the definition of a crossed homomorphism and prove our generalization of BKV theorem in the section ``Crossed homomorphisms" (this proof is not a novelty, since we basically repeat the reasoning of \cite{Klyachko_2017} and add Lemma~\ref{lem:tail}, which seems to be a known result, though we could not find a clear reference). We prove the special case of the Asai-Yoshida conjecture in the section ``Main theorem" (our proof follows the ideas of \cite{Yoshida_1993}). We show how the result from the abstract (for finite groups \(F\) and \(G\) the number \(|\Hom (F, G)|\) is divisible by \(\gcd(|G|, |F : F'|)\) if \(F/F'\) is a product of a cyclic group and a group with cube-free exponent) follows from our theorems in the section ``The number of homomorphisms''.

\textbf{Notation and Conventions.} Throughout the paper, all actions are assumed to be right actions, and the image of a point \(x\) under the action of a group element \(g\) is denoted \(x^g\). For group elements, as usual, 
\(g^h = h^{-1}gh\). If a group \(M\) acts on a group \(G\), the semidirect product of these groups is denoted \(M \ltimes G\).

A cyclic group in additive notation is denoted \(\Z\) (in the case of an infinite group) or \(\Z/n\Z\) (in the case of a group of order \(n\)), and in multiplicative notation it is denoted \(\langle a \rangle_{\infty}\) (in the case of an infinite cyclic group generated by an element \(a\)) or \(\langle a \rangle_n\) (in the case of a cyclic group of order \(n\) generated by an element \(a\)).

\(Z(G)\), \(G'\), \(|G|\)  denote the center, the commutator subgroup and the order of a group \(G\) respectively, and \(|X|\) denotes the number of elements of a set \(X\). \(\Hom(F, G)\) denotes the set of all group homomorphisms \(F \to G\), and \(|G : H|\) denotes the index of a subgroup \(H\) in a group \(G\). The least common multiplier and the greatest common divisor are denoted \(\lcm\) and \(\gcd\) respectively.

Finiteness of groups is not assumed anywhere by default; divisibility is always understood in the sense of cardinal arithmetic (an infinite cardinal is divisible by all non-zero cardinals not exceeding it, and zero is divisible by all cardinals). For a group \(G\) and an integer \(n\) their \(\gcd\) is defined by \(\gcd(G, n) = \lcm(\{|H| : \text{\(H\) is a subgroup of \(G\) and \(|H|\) divides \(n\)}\})\) (this makes sense when \(G\) is infinite and when \(G\) is finite this coincides with \(\gcd(|G|, n)\) due to Sylow theorems). 

The author is grateful to his scientific advisor Anton A.~Klyachko. The author thanks the Theoretical Physics and Mathematics Advancement Foundation ``BASIS".

\section{Crossed homomorphisms}

\begin{definition}\label{def:crossed}
    Let \(M\) be a group acting on a group \(H\). A map \(\alpha \colon M \to H\) is called a \textit{crossed homomorphism}, if a map \(M \to M \ltimes H\) given by \(a \mapsto (a, \alpha(a))\) is a homomorphism. Equivalently, \(\alpha\) is a crossed homomorphism if \(\alpha(ab) = \alpha(a)^b \alpha(b)\) for all \(a, b \in M\).
\end{definition}

\begin{example}
    If the action of \(M\) is trivial, then the crossed homomorphisms are just the usual homomorphisms.    
\end{example}

\begin{example}\label{exmpl:infinite_cyclic}
    Let \(M\) be an infinite cyclic group \(\langle a \rangle_\infty\). Then, for each element \(h \in H\), there is an unique homomorphism defined by \(a \mapsto (a, h)\), hence an unique crossed homomorphism defined by \(a \mapsto h\), and crossed homomorphisms \(M \to H\) are in one-to-one correspondence with elements of \(H\).
\end{example}

\begin{example}\label{exmpl:finite_cyclic}
    Let \(M\) be a finite cyclic group \(\langle a \rangle_n\) which order is a multiple of \(|H|\). Then a crossed homomorphism is uniquely determined by an element \(h \in H\) such that \((a, h)^n = 1\) in \(M \ltimes H\). It is not trivial, but it follows from the observation of Brauer (which we state in the next paragraph) that under given conditions \((a, h)^n = 1\) for all \(h \in H\). It follows that, just as in the infinite cyclic case, the crossed homomorphisms \(M \to H\) are in one-to-one correspondence with elements of \(H\).
\end{example}

\begin{brauer_lemma}
    If \(H\) is a finite normal subgroup of a group \(G\), then, for every \(a \in G\) and \(h \in H\), the elements \(a^{|H|}\) and \((ah)^{|H|}\) are conjugate by an element of \(H\).
\end{brauer_lemma}

One can find Brauer's elegant proof of this lemma in \cite{Brauer_1969} or \cite{Brusyanskaya_2019}. We will not use this lemma any further.

\begin{example}\label{exmpl:non_abelian}
    Suppose that \(M\) is a non-abelian finite group and \(H\) is a cyclic group of order \(|M|\). Then the number of (non-crossed) homomorphisms \(M \to H\) is equal to \(|M : M'|\), which is strictly lesser than \(|H|\).
\end{example}

From our point of view, the purpose of introducing crossed homomorphisms is the following Lemma~\ref{lem:tail} (the definitions are due to \cite{Klyachko_2017}).

Suppose that \(\phi \colon F \to G\) is a group homomorphism. Fix a subgroup \(H\) of the group \(G\) and an epimorphism \(\deg\) from the group \(F\) onto some group \(M\). 

We define the \textit{tail} of \(\phi\) (with respect to a quadruple \((F \epi M; H \subset G)\)) to be a pair \((\phi_0, \phi_H)\), where \(\phi_0\) is the restriction of the homomorphism \(\phi\) to the subgroup \(\ker \deg \subset F\), and \(\phi_H\) is the mapping from \(F\) into the set of left cosets of the group \(G\) with respect to the subgroup \(H\), which sends an element \(f \in F\) to the coset \(\phi(f) H\).

The subgroup of \(H\) given by
\[ H_\phi = \bigcap_{f \in F} H^{\phi(f)} \cap C(\phi(\ker \deg))\]
is called the \textit{\(\phi\)-core} of H. In other words, the \(\phi\)-core \(H_\phi\) of \(H\) consists of elements \(h\) such that \(h^{\phi(f)} \in H\) for all \(f \in F\), and \(h^{\phi(f)} = h\) if \(f \in \ker \deg\). Note that the group \(M\) acts on \(H_\phi\) by the rule \(h^a = h^{\phi(f)}\) for \(a \in M\), where \(f \in F\) is an arbitrary element such that \(\deg f = a\).

\begin{lemma}\label{lem:tail}
    The homomorphisms having the same tail as \(\phi\) are precisely the homomorphisms of the form \(f \mapsto \phi(f) \alpha(\deg f)\), where \(\alpha\) is a crossed homomorphism \(M \to H_\phi\). In particular, the number of homomorphisms having the same tail as \(\phi\) coincides with the number of crossed homomorphisms \(M \to H_\phi\).
\end{lemma}

\begin{proof}

For a mapping \(\alpha \colon M \to H_\phi\), the rule \(f \mapsto \phi(f)\alpha(\deg f)\) defines a homomorphism iff \(\alpha\) is a crossed homomorphism. Indeed, it defines a homomorphism when
\[ \phi(fg) \alpha(\deg fg) = \phi(f) \alpha(\deg f) \phi(g) \alpha(\deg g), \]
and that is equivalent to \(\alpha(\deg fg) = \alpha(\deg f)^{\phi(g)} \alpha(\deg g)\). Since \(\deg\) is an epimorphism, we can pass to elements of \(M\) and get \(\alpha(ab) = \alpha(a)^b \alpha(b)\) for \(a, b \in M\), which is the definition of a crossed homomorphism. 

It remains to prove that any homomorphism \(\psi \colon F \to G\) having the same tail as \(\phi\) (that is, \(\phi(f)^{-1}\psi(f)\) lies in \(H\) for \(f \in F\) and equals \(1\) if \(f \in \ker \deg\)) is given by that rule (i.e. \(\psi \colon f \mapsto \phi(f)\alpha(\deg f)\) for some \(\alpha \colon M \to H_\phi\)). Indeed, \(\phi(f)^{-1} \psi(f)\) only depends on \(\deg f\), since for \(f' \in \ker \deg\) one has
\[\phi(f'f)^{-1}\psi(f'f) = \phi(f)^{-1}\phi(f')^{-1}\psi(f')\psi(f) = \phi(f)^{-1}\psi(f).\]
A subgroup of \(H\) generated by a set \(\{\phi(f)^{-1} \psi(f) : f \in F\}\) is invariant under every conjugation \(h \mapsto h^{\phi(f')}, f' \in F\), and each of its elements is fixed by every conjugation \(h \mapsto h^{\phi(f')}, f' \in \ker \deg\), because
\[\left(\phi(f)^{-1} \psi(f)\right)^{\phi(f')} = \left(\phi(ff')^{-1}\psi(ff')\right)\left(\phi(f')^{-1}\psi(f')\right)^{-1}.\]
It follows that this subgroup lies in \(H_\phi\). Defining \(\alpha\) by \(\alpha(\deg f) = \phi(f)^{-1} \psi(f)\) completes the proof.

\end{proof}

The following Theorem~\ref{thm:hom_family} connects certain divisibility problems in groups with this statement, which is equivalent to the Asai-Yoshida conjecture.

\begin{restricted_asai_yoshida_conjecture}
    The number of crossed homomorphisms from a finitely generated abelian group \(M\) to a group \(H\) on which \(M\) acts is divisible by \(|H|\) provided that \(|H|\) divides the order of \(M\).
\end{restricted_asai_yoshida_conjecture}

Examples~\ref{exmpl:infinite_cyclic} and~\ref{exmpl:finite_cyclic} show that for cyclic \(M\) the conjecture holds, and example~\ref{exmpl:non_abelian} shows that we cannot get rid of the commutativity of \(M\).

A group \(F\) equipped with an epimorphism onto a finitely generated abelian group \(M\) we call an \textit{\(M\)-indexed} group. This epimorphism \(F \to M\) we call \textit{degree} and denote \(\deg\).

\begin{theorem}\label{thm:hom_family}

Suppose that \(M\) is a finitely generated abelian group, \(H\) is a subgroup of a group \(G\), and a set \(\Phi\) of homomorphisms from an \(M\)-indexed group \(F\) to \(G\) satisfies the following conditions.
\begin{enumerate}

    \item \(\Phi\) is invariant with respect to conjugation by elements of \(H\).

    \item For any \(\phi \in \Phi\), each homomorphism \(\psi\) having the same tail as \(\phi\) with respect to a quadruple \((F \epi M; H \subset G)\) belongs to \(\Phi\) too.
    
\end{enumerate}
Suppose also that the number of crossed homomorphisms \(M \to H^*\) is divisible by \(|H^*|\) for each subgroup \(H^*\) of \(H\) and every action of the group \(M\) on \(H^*\).

Then \(|\Phi|\) is divisible by \(|H|\).

\end{theorem}

When we take \(F =F_0, M = M_0, G = G_0, H = H_0\) in Theorem~\ref{thm:hom_family}, we say that we apply that theorem \textit{to the quadruple \((F_0 \epi M_0; H_0 \subset G_0)\)}.

\begin{proof}

We say that two homomorphisms \(\phi, \psi \in \Phi\) are \textit{similar} and write \(\phi \sim \psi\), if their tails are conjugate by an element of \(H\), that is
\begin{multline*}
\phi \sim \psi \Longleftrightarrow \text{there exists \(h \in H\) such that} \\
\begin{aligned}
&\phi(f) = h\psi(f)h^{-1} &&\text{for each \(f \in F\) of degree \(0\), and} \\
&\phi(f)H = h\psi(f)H &&\text{for each \(f \in F\).} 
\end{aligned}
\end{multline*} 

Similarity is clearly an equivalence relation on \(\Phi\). The theorem is an immediate consequence of the following claim.

\begin{claim*} 

The number of elements in each class of similar homomorphisms is a multiple of \(|H|\). More precisely, for each \(\phi \in \Phi\) 

\begin{enumerate}
    \item the number of distinct tails of homomorphisms in \(\Phi\) similar to \(\phi\) equals \(|H : H_\phi|\);
    \item for each homomorphism \(\psi\) which is similar to \(\phi\) the number of homomorphisms in \(\Phi\) having the same tail as \(\psi\) does not depend on \(\psi\) and is divisible by \(|H_\phi|\).
\end{enumerate}

\end{claim*}

To prove the first part of the claim, one can note that the group \(H\) acts by conjugation on the set of the tails of homomorphisms in \(\Phi\) (due to the first condition of the theorem). Tails of homomorphisms similar to \(\phi\) form the orbit of the tail of the homomorphism \(\phi\) under this action. The cardinality of the orbit is known to be equal to the index of the stabilizer. It remains to note that the subgroup \(H_\phi\) is the stabilizer of the tail of the homomorphism \(\phi\).

The second part of the claim follows from Lemma~\ref{lem:tail}. The set of homomorphisms having the same tail as \(\psi\) and the set of homomorphisms having the same tail as \(\phi\) are conjugate by an element of \(H\), hence have the same cardinality. Due to Lemma~\ref{lem:tail} this cardinality equals the number of crossed homomorphisms \(M \to H_\phi\). This number is divisible by \(|H_\phi|\) by the assumption of the theorem.
\end{proof}

\begin{observation}\label{obs:hom_family_relative}
    The proof uses the condition on the number of crossed homomorphisms \(M \to H^*\) only for \(H^* = H_\phi\). Therefore this condition can be replaced, for instance, to an assumption that the restricted Asai-Yoshida conjecture holds for \(M\), and that the order of each subgroup \(H_\phi\) divides \(|M|\).
\end{observation}

\section{Some corollaries}

We start by introducing a more general version of Theorem~\ref{thm:hom_family}. Recall that, for a group \(G\) and an integer \(n\), their greatest common divisor \(\gcd(G, n)\) is given by
\(\gcd(G, n) = \lcm(\{|H| : \text{\(H\) is a subgroup of \(G\) and \(|H|\) divides \(n\)}\})\).

\begin{theorem}\label{thm:hom_family_extended}

Suppose that \(M\) is a finitely generated abelian group, \(H\) is a subgroup of a group \(G\), and a set \(\Phi\) of homomorphisms from an \(M\)-indexed group \(F\) to \(G\) satisfies the following conditions.
\begin{enumerate}

    \item \(\Phi\) is invariant with respect to conjugation by elements of \(H\).

    \item For any \(\phi \in \Phi\), each homomorphism \(\psi\) having the same tail as \(\phi\) with respect to the quadruple \((F \epi M; H \subset G)\) belongs to \(\Phi\) too.
    
\end{enumerate}
Suppose also that the number of crossed homomorphisms \(M \to H^*\) is divisible by \(|H^*|\) for each subgroup \(H^*\) of \(H\) and every action of the group \(M\) on \(H^*\) provided that \(|H^*|\) divides the order of \(M\).

Then \(|\Phi|\) is divisible by \(\gcd(H, |M|)\).

\end{theorem}

\begin{proof}
    Apply Theorem~\ref{thm:hom_family} to all quadruples \((F \epi M; H^* \subset G)\) such that \(H^*\) is a subgroup \(H\) and \(|H^*|\) divides \(|M|\) and use the definition of \(\gcd\).
\end{proof}

When we take \(F = F_0, M = M_0, G = G_0, H = H_0\) in Theorem~\ref{thm:hom_family_extended}, we say that we apply that theorem \textit{to the quadruple \((F_0 \epi M_0; H_0 \subset G_0)\)}.

The following corollaries are known results, though we include proofs to demonstrate the strength of theorems~\ref{thm:hom_family} and~\ref{thm:hom_family_extended}.

\begin{corollary}[restricted Asai-Yoshida conjecture implies its full version]\label{corol:crossed}
    Suppose that \(M\) is a finitely generated abelian group acting on a group \(H\) such that, for each subgroup \(H^*\) of \(H\) which order divides the order of \(M\) and every action of the group \(M\) on \(H^*\), the number of crossed homomorphisms \(M \to H^*\) is divisible by \(|H^*|\).

    Then the number of crossed homomorphisms from the group \(M\) to \(H\) is divisible by \(\gcd(H, |M|)\).
\end{corollary}

\begin{proof}
    Note that crossed homomorphisms \(M \to H\) correspond to homomorphisms \(M \to M \ltimes H\) such that \(a \mapsto (a, h_a)\) for all \(a \in M\) and some \(h_a \in H\). It remains to apply Theorem~\ref{thm:hom_family_extended} to this family of homomorphisms and the quadruple \((M \epi M; H \subset M \ltimes H)\).
\end{proof}

\begin{corollary}[reduction to \(p\)-groups]
    Suppose that \(M\) is a finite abelian group acting on a group \(H\) such that, for each \(p\)-torsion \(M_p\) of the group \(M\), each subgroup \(H^*\) of \(H\) which order divides \(|M_p|\) and every action of the group \(M\) on \(H^*\), the number of crossed homomorphisms \(M_p \to H^*\) is divisible by \(|H^*|\).

    Then the number of crossed homomorphisms from the group \(M\) to \(H\) is divisible by \(\gcd(H, |M|)\).
\end{corollary}

\begin{proof}
    Apply Theorem~\ref{thm:hom_family_extended} to each quadruple \((M \epi M_p; H_p \subset M \ltimes H)\), where \(H_p\) is a subgroup of \(H\) such that \(|H_p| = \gcd(H, |M_p|)\). It remains to note that, due to Sylow theorems, \(\gcd(H, |M|) = \lcm(\{|H_p| : \text{\(p\) is prime}\})\).
\end{proof}

The following corollary, together with our Theorem~\ref{thm:main}, yields that, for finite groups \(F\) and \(G\), \(|\Hom(F, G)|\) is divisible by \(\gcd(|G|, |F : F'|)\) if \(F/F'\) is a product of a cyclic group and a group with cube-free exponent.

\begin{corollary}\label{corol:annotated}
    Suppose that \(F\) and \(G\) are finite groups such that, for each \(p\)-torsion \((F/F')_p\) of \(F/F'\) and each subgroup \(H\) of \(G\) on which \((F/F')_p\) acts, the number of crossed homomorphisms \(M_p \to H\) is divisible by \(|H|\) provided that \(|H|\) divides \(|(F/F')_p|\). Then the number of homomorphisms \(F \to G\) is divisible by \(\gcd(|G|, |F : F'|)\).
\end{corollary}

\begin{proof}
    Use the previous corollary to conclude that the number of crossed homomorphisms \(F/F' \to H\) is divisible by \(\gcd(H, |F/F'|)\) for every subgroup \(H\) of \(G\), then apply Theorem~\ref{thm:hom_family_extended} to the quadruple \((F \epi F/F'; G \subset G)\).
\end{proof}

\section{Main theorem}

We aim to prove the following theorem (Asai-Yoshida conjecture for \(\Z/p^n\Z \times (\Z/p\Z)^m \times  (\Z/p^2\Z)^k\)).

\begin{theorem}\label{thm:main}
    Suppose that \(M = \Z/p^n\Z \times (\Z/p\Z)^m \times  (\Z/p^2\Z)^k\), where p is a prime, and \(H\) is a group on which \(M\) acts. Then the number of crossed homomorphisms \(M \to H\) is divisible by \(\gcd(H, |M|)\).
\end{theorem}

However, due to Corollary~\ref{corol:crossed}, we only need to prove its restricted version.

\begin{theorem}
    Suppose that \(M = \Z/p^n\Z \times (\Z/p\Z)^m \times  (\Z/p^2\Z)^k\), where p is a prime, and \(H\) is a group on which \(M\) acts. Then the number of crossed homomorphisms \(M \to H\) is divisible by \(|H|\) provided that \(|H|\) divides \(|M|\).
\end{theorem}

We are going to use the fact that crossed homomorphisms \(M \to H\) correspond to homomorphisms \(M \to M \ltimes H\) such that \(a \mapsto (a, h_a)\) for all \(a \in A\) and some \(h_a \in H\) (this way it is possible to analyze the number of crossed homomorphisms using Theorem~\ref{thm:hom_family}). Such homomorphisms \(M \to M \ltimes H\) we call \textit{sections \(M \to M \ltimes H\)}.

Our proof uses induction on the size of the group \(M\). If \(M\) is cyclic, then the number of crossed homomorphisms \(M \to H\) is precisely \(|H|\) due to the example~\ref{exmpl:finite_cyclic}. Otherwise, we can assume that \(M = M_0 \times \Z/p^s\Z\), where \(s \in \{1, 2\}\), and that the induction hypothesis holds for each proper quotient group of \(M\).

Start with an observation that, if the subgroup \(Z_H\) consisting of all elements of \(H\) which centralize both \(H\) and \(M\) (that is, \(Z_H = H \cap Z(M \ltimes H)\)) is large enough, then we can prove the induction step as follows, via ``shifting" our crossed homomorphisms by elements of \(\Hom(\Z/p^s\Z, Z_H)\).

\begin{lemma}\label{lem:large_center}
    Suppose that \(M = M_0 \times  \Z/p^s\Z\), the induction hypothesis holds for \(M_0\) and the order of the subgroup \(Z_H = H \cap Z(M \ltimes H) \subset H\) is divisible by \(p^s\). Then the number of sections \(M \to M \ltimes H\) (which is equal to the number of crossed homomorphisms \(M \to H\)) is divisible by \(|H|\).
\end{lemma}

\begin{proof}

    Consider tails of sections \(\phi \colon M \to M \ltimes H\) with respect to the quadruple \((M \epi M_0; H \subset M \ltimes H)\), that is, pairs \((\phi_0, \phi_H)\) where \(\phi_0 = \phi|_{\ker \deg} = \phi|_{\Z/p^s\Z}\), and \(\phi_H \colon M \to (M \ltimes H)/H\) is given by \(\phi_H{:}\; a \mapsto \phi(a)H\). 

    Note that the condition \(\psi_H = \phi_H\) implies that any homomorphism \(M \to M \ltimes H\) having the same tail as a section \(\phi\) is also a section. The number of sections having the same tail as \(\phi\) therefore equals the number of crossed homomorphisms \(M_0 \to H_\phi\) (by Lemma~\ref{lem:tail}), and hence is divisible by \(\gcd(H_\phi, |M_0|)\).

    As in the proof of Theorem~\ref{thm:hom_family}, \(H\) acts on the tails of sections by conjugation, and the number of tails conjugate to a tail of a section \(\phi\) is precisely \(|H : H_\phi|\).

    On the other side, since the group \(Z_H \subset H\) centralizes \(M \ltimes H\), one can shift an arbitrary section \(\phi \colon (M = M_0 \times \Z/p^s\Z) \to M \ltimes H\) by an arbitrary homomorphism \(\alpha \colon \Z/p^s\Z \to Z_H\) to obtain another section given by
    \[ M_0 \times \Z/p^s\Z \ni (a, b) \mapsto \phi(a,b) \cdot \alpha(b) \in M \ltimes H. \]
    Those shifts clearly define an action of the group \(\Hom(\Z/p^s\Z, Z_H)\) on the tails of sections, and the length of any orbit is equal to \(|\Hom(\Z/p^s\Z, Z_H)|\), because only the trivial element of \(\Hom(\Z/p^s\Z, Z_H)\) fixes \(\phi_0 = \phi|_{\Z/p^s\Z}\). On the other side, \(|\Hom(\Z/p^s\Z, Z_H)|\) is divisible by \(p^s\), because \(|Z_H| \ge p^s\). 

    It remains to note that the action of \(\Hom(\Z/p^s\Z, Z_H)\) commutes with the conjugation action of \(H\), so we can combine those actions to get an action of \(H \times \Hom(\Z/p^s\Z, Z_H)\). Length of an orbit of a tail of a section \(\phi\) with respect to this action is divisible by \(\lcm(|H : H_\phi|, p^s)\), and the number of sections having tail in such an orbit is a multiple of
    \[\gcd(|H_\phi|, |M_0|) \cdot \lcm(|H : H_\phi|, p^s)\]
    which is divisible by \(\gcd(|H|, |M|) = |H|\).

    This yields a subdivision of the set of all sections \(M \to M \ltimes H\) into orbits, each of which has a size divisible by \(|H|\).
    
\end{proof}

Lemma~\ref{lem:large_center} in particular proves the induction step for \(M = M_0 \times \Z/p\Z\), since \(Z_H = H \cap Z(M \ltimes H)\) is nontrivial unless \(H\) is trivial (as an intersection of the center and a normal subgroup in a finite \(p\)-group).

The proof of the induction step for \(M = M_0 \times \Z/p^2\Z\) goes as follows. Switch to multiplicative notation, so \(M = M_0 \times \langle a \rangle_{p^2}\). We consider a larger group \(\hat M = M_0 \times \langle \hat a \rangle_\infty\), and identify \(M\) with \(\hat M/\langle \hat a^{p^2} \rangle\). Then we embed the set of all sections \(M \to M \ltimes H\) into a set
\begin{multline*}
    \Phi = \{\phi \colon \hat M \to M \ltimes H : \text{\(\phi\) is a homomorphism,} \\ \text{and \(\phi(x) = (\pi(x), h_x)\) for all \(x \in \hat M\) and some \(h_x \in H\)}\}
\end{multline*}
(where \(\pi\) is the projection of \(\hat M\) onto \(M\)) as a subset 
\[\hat \Phi = \{\phi \in \Phi : \phi(\hat a^{p^2}) = 1\}\]
It follows from Theorem~\ref{thm:hom_family} applied to the quadruple \((\hat M \epi \langle \hat a \rangle_\infty; H \subset M \ltimes H)\) (that is, BKV theorem in the case of a \(0\)-indexed group \(\hat M\)) that \(|\Phi|\) is divisible by \(|H|\). We show that either Observation~\ref{obs:hom_family_relative} forces the number of homomorphisms from \(\Phi\) \textit{which do not belong to \(\hat \Phi\)} to be divisible by \(|H|\), or it is possible to obtain a cyclic subgroup of \(Z_H = H \cap Z(M \ltimes H)\) of order \(p^2\), and then Lemma~\ref{lem:large_center} completes the proof.

We divide our argument into three lemmata.

\begin{lemma}
    Under the above notation \(|\Phi|\) is finite and divisible by \(|H|\).
\end{lemma}

\begin{proof} 
    Family \(\Phi\) satisfies the conditions of Theorem~\ref{thm:hom_family} applied to the quadruple \((\hat M \epi \langle \hat a \rangle_{\infty}; H \subset M \ltimes H)\), hence \(|\Phi|\) is divisible by \(|H|\). Finiteness is obvious since \(\Phi\) is a set of homomorphisms from a finitely generated group to a finite group.
\end{proof}

\begin{lemma}
    Either the number of sections \(M \to M \ltimes H\) is divisible by \(|H|\), or there is a section \(\phi\) such that \(\phi(a^p) \in Z(M \ltimes H)\).
\end{lemma}

\begin{proof}
    Consider the quadruple \((M \epi M/\langle a^p \rangle; H \subset M \ltimes H)\). Since \(|H|\) divides \(|M|\), the order of any proper subgroup of \(H\) divides \(|M|/p = |M/\langle a^p \rangle|\). Hence, and also because the induction hypothesis holds for \(M/\langle a^p \rangle\), we can apply Observation~\ref{obs:hom_family_relative} to the set of all sections \(M \to M \ltimes H\), unless there is a section \(\phi\) for which \(H_\phi = H\). But \(H_\phi\) is centralized by \(\phi(\ker \deg) = \phi(\langle a^p \rangle)\), and if \(H_\phi = H\) then \(\phi(a^p)\) centralizes \(\phi(M) \cdot H = M \ltimes H\).
\end{proof}

\begin{lemma}
    Either the number of elements of \(\Phi \setminus \hat \Phi\) is divisible by \(|H|\), or there is a homomorphism \(\psi \in \Phi \setminus \hat \Phi\) such that \(\psi(\hat a^p) \in Z(M \ltimes H)\).
\end{lemma}

\begin{proof}
    The same as the proof of the previous lemma.
\end{proof}

Either the number of sections itself is divisible by \(|H|\), or \(|\hat \Phi| = |\Phi| - |\Phi \setminus \hat \Phi|\) is divisible by \(|H|\), or there exists an element \(x = \phi(a^p) \in Z(M \ltimes H)\) of order \(p\) and an element \(y = \psi(\hat a^p) \in Z(M \ltimes H)\) of order at least \(p^2\) (we assumed that \(\psi \notin \hat \Phi\), so \(\psi(\hat a^{p^2}) \neq 1\)) such that \(x\) and \(y\) have the same image under the projection \(M \ltimes H \to M\). It follows that \(xy^{-1}\) is an element of order at least \(p^2\) lying in \(Z_H = H \cap Z(M \ltimes H)\), and then Lemma~\ref{lem:large_center} can be applied. The number of crossed homomorphisms is divisible by \(|H|\) in each of these cases, which completes the proof.

\begin{remark}\label{rem:induction_step}
    We only needed the induction hypothesis in the form ``for any proper quotient group \(M_0\) of \(M\) and any subgroup \(H^*\) of \(H\) such that \(|H^*|\) divides \(|M_0|\), the number of crossed homomorphisms \(M_0 \to H\) is divisible by \(|H|\)".
    
    Our induction step for \(M = M_0 \times \langle a \rangle_{p^2}\) also works for \(M = M_0 \times \langle a \rangle_{p^s}\), \(s > 2\) if \(H\) does not contain elements of order \(p^2\) that centralize \(M \ltimes H\) (one should replace \(a^p\) in the proof with \(a^{p^{s - 1}}\), then the three lemmata above remain the same, and instead of using Lemma~\ref{lem:large_center} we note that the existence of an element of order at least \(p^2\) lying in \(Z_H\) leads to a contradiction).
\end{remark}

This way we can obtain the following theorem.

\begin{theorem}
    Let \(H\) be a finite \(p\)-group such that every non-abelian subgroup of \(H\) has an elementary abelian center. Then, for each finite abelian \(p\)-group \(M\) acting on \(H\), the number of crossed homomorphisms \(M \to H\) is divisible by \(\gcd(|H|, |M|)\).
\end{theorem}

\begin{proof}
    For abelian \(H\) see \cite{ASAI_TAKEGAHARA_1999}. If \(H\) is non-abelian, apply the induction argument as it was said in Remark~\ref{rem:induction_step}.
\end{proof}

\section{The number of homomorphisms}
\begin{theorem}\label{thm:annotated}
    For finite groups \(F\) and \(G\) the number \(|\Hom (F, G)|\) is divisible by \(\gcd(|G|, |F : F'|)\) if \(F/F'\) is a product of a cyclic group and a group with cube-free exponent.
\end{theorem}
\begin{proof}
    Theorem~\ref{thm:main} shows that for each \(p\)-torsion \((F/F')_p\) of \(F/F'\) the number of crossed homomorphisms to any finite \(p\)-group \(H\) is divisible by \(\gcd(|F/F'|, |H|)\) (i.e. the Asai-Yoshida conjecture holds in that case). Corollary~\ref{corol:annotated} then implies that the number of homomorphisms \(F \to G\) is divisible by \(\gcd(|G|, |F : F'|)\).
\end{proof}

\bibliographystyle{alpha}
\bibliography{bibliography}

\end{document}